\documentclass[a4paper,12pt]{article}
\usepackage{latexsym}
\usepackage{amsmath}
\usepackage{amssymb}
\usepackage{enumerate}

\usepackage{theorem}
\newtheorem{theorem}{Theorem}[section]
\newtheorem{proposition}[theorem]{Proposition}
\newtheorem{lemma}[theorem]{Lemma}
\newtheorem{corollary}[theorem]{Corollary}

\newtheorem{MT}{Main Theorem}
\newtheorem{TCT}{A New Type of Toponogov  Comparison Theorem}

\theorembodyfont{\rmfamily}
\newtheorem{proof}{\textmd{\textit{Proof.}}}

\newtheorem{remark}[theorem]{Remark}

\newtheorem{acknowledgement}{\textmd{\textit{Acknowledgements.}}}

\makeatletter

\@addtoreset{equation}{section}
\makeatother

\newcommand{\qedd}{\hfill \Box}
\newcommand{\ve}{\varepsilon}
\newcommand{\lra}{\longrightarrow}
\newcommand{\wt}{\widetilde}
\newcommand{\wh}{\widehat}
\newcommand{\ol}{\overline}

\newcommand{\B}{\ensuremath{\mathbb{B}}}
\newcommand{\N}{\ensuremath{\mathbb{N}}}
\newcommand{\R}{\ensuremath{\mathbb{R}}}
\newcommand{\Sph}{\ensuremath{\mathbb{S}}}

\newcommand{\cE}{\ensuremath{\mathcal{E}}}

\newcommand{\cR}{\ensuremath{\mathcal{R}}}

\newcommand{\cU}{\ensuremath{\mathcal{U}}}

\def\diam{\mathop{\mathrm{diam}}\nolimits}

\title
{
Total Curvatures of Model Surfaces Control\\ 
Topology\hspace{-0.5mm} of Complete Open Manifolds with\\ 
Radial Curvature Bounded Below.\,III\footnote{
Mathematics Subject Classification (2000)\,:\,53C20, 53C21.}
\footnote{
Keywords\,:\,Busemann function, radial curvature, total curvature
}
}
\author{Kei KONDO $\cdot$ Minoru TANAKA}
\date{}
\begin{document}
\maketitle

\vspace{-7.7mm}
\begin{center}
{\em \footnotesize Dedicated to Professor K. Shiohama on the occasion of his seventieth birthday.}
\end{center}

\begin{abstract}
This article is the third in a series of our investigation on a complete 
non-compact connected Riemannian manifold $M$. In the first series \cite{KT1}, 
we showed that all Busemann functions on an $M$ which is not less curved than a von Mangoldt 
surface of revolution $\wt{M}$ are exhaustions, if the total curvature of $\wt{M}$ is 
greater than $\pi$. A von Mangoldt surface of revolution is, by definition, a complete surface of revolution homeomorphic to $\R^{2}$ whose Gaussian curvature is non-increasing along each meridian. 
Our purpose of this series is to generalize the main theorem in \cite{KT1} to an $M$ which is 
not less curved than a more general surface of revolution.
\end{abstract}

\section{Introduction}\label{sec:intro}
The Gauss--Bonnet theorem says that 
the total curvature $c(S)$ of a {\bf compact} Riemannian $2$-dimensional manifold $S$ 
is a topological invariant, i.e., 
\[
c(S) = 2 \pi \chi (S).
\] 
Here $\chi (S)$ denotes the Euler characteristic of $S$.\par 
In 1935, Cohn\,-Vossen generalized the Gauss--Bonnet theorem for complete 
{\bf non-compact} Riemannian $2$-dimensional manifolds as follows\,:

\begin{theorem}{\rm (\cite[Satz 6]{CV1})}\label{CVthm}
If a connected, complete non-compact, finitely connected Riemannian 
$2$-manifold $M$ admits a total curvature $c(M)$, then,
\[
c(M) \le 2 \pi \chi (M)
\]
holds. Here $\chi (M)$ denotes the Euler characteristic of $M$.
\end{theorem}

\medskip\noindent
Notice the total curvature $c(M)$ is not a topological invariant anymore. 
But $2 \pi \chi (M) - c(M)$ is a geometric invariant depending only on the ends of $M$, which is 
a consequence from the isoperimetric inequalities (see \cite[Theorem 5.2.1]{SST}).

\bigskip

In 1984, Shiohama proved the next result peculiar to geometry of total curvature on surfaces\,:

\begin{theorem}{\rm (\cite[Main Theorem]{Sh1})}\label{Sh1}
Let $M$ be a connected, complete non-compact, finitely connected and 
oriented Riemannian $2$-manifold with one end. 
If the total curvature $c (M)$ satisfies 
\[
c(M) > (2 \chi (M) - 1) \pi,
\]
then all Busemann functions on $M$ are exhaustions. 
In particular, if the total curvature of $M$ is greater than $\pi$, 
then $M$ is homeomorphic to $\R^2$ 
and also all Busemann functions are exhaustions. 
\end{theorem} 

\bigskip\noindent
Here the {\em Busemann function} $F_{\gamma}:M \lra \R$ of a ray $\gamma$ in 
a complete non-compact Riemannian (any dimensional) manifold $M$ is, by definition, 
\[
F_{\gamma}(x) := \lim_{t \to \infty} \left\{ t - d(x, \gamma(t)) \right\}_,
\]
and a function $\varphi : M \lra \R$ is called an exhaustion, 
if $\varphi^{-1}( - \infty, a ]$ is compact for all $a \in \R$. 

\bigskip\noindent

Theorem \ref{Sh1} was generalized to higher-dimensional manifolds in \cite{KT1}. 
Roughly speaking, it was proved in \cite{KT1} that all Busemann functions on a complete 
non-compact connected Riemannian manifold not less curved than a von Mangoldt 
surface of revolution $\wt{M}$ are exhaustions, if the total curvature of $\wt{M}$ is 
greater than $\pi$ (The theorem will be later stated in full detail as Theorem \ref{thm1.4} in this article). 
\par 
A von Mangoldt surface of revolution is, by definition, a complete surface of revolution homeomorphic to $\R^{2}$ whose Gaussian curvature is non-increasing along each meridian. The monotonicity of 
the Gaussian curvature of a von Mangoldt surface of revolution looks restrictive, but very familiar surfaces such as a paraboloid or a $2$-sheeted hyperboloid are von Mangoldt surfaces of 
revolution.

\bigskip

Although Cohn\,-Vossen restricted himself to $2$-dimensional manifolds, 
he has developed fundamental techniques, 
such as drawing a circle or a geodesic polygon, and joining two points 
by a minimal geodesic segment, to investigate 
the structures of complete Riemannian $2$-dimensional manifolds. 
We, Riemannian geometers, should be awed by the fact that 
such techniques are ever now not only useful, but also powerful 
for investigating the topology of {\bf any dimensional} complete Riemannian manifolds.\par 
Furthermore, as pointed out in the preface of \cite{SST}, 
{\em 
it took more than thirty years to obtain higher-dimensional extensions of 
Cohn\,-Vossen's results} for complete non-compact Riemannian $2$-dimensional manifolds. 
They are the splitting theorem by Toponogov \cite{To}, 
the structure theorem with positive sectional curvature by Gromoll and Meyer \cite{GM}, 
and the soul theorem with non-negative sectional curvature by Cheeger and Gromoll \cite{CG}. 
Hence, it requires many years and is also very difficult to generalize some results peculiar to geometry 
of surfaces to any dimensional complete Riemannian manifolds. 
In fact, one may find such results in \cite{SST}, which have not been generalized in higher 
dimensions yet.

\bigskip

Our purpose of this article is to generalize the main theorem in \cite{KT1} to a complete 
non-compact connected Riemannian manifold not less curved than a more general 
surface of revolution. To state this precisely, we will begin on the definition of a non-compact 
model surface of revolution. 

\bigskip
 
Let $\wt{M}$ denote a complete $2$-dimensional Riemannian manifold 
homeomorphic to $\R^{2}$ with a base point $\tilde{p} \in \wt{M}$. 
Then, we call the pair $(\wt{M}, \tilde{p})$ a {\em non-compact model surface of revolution} 
if its Riemannian metric $d\tilde{s}^2$ is expressed 
in terms of geodesic polar coordinates around $\tilde{p}$ as 
\begin{equation}\label{polar}
d\tilde{s}^2 = dt^2 + f(t)^2d \theta^2, \quad 
(t,\theta) \in (0,\infty) \times {\Sph_{\tilde{p}}^1}_. 
\end{equation}
Here $f : (0, \infty) \lra \R$ is a positive smooth function 
which is extensible to a smooth odd function around $0$, 
and 
$\Sph^{1}_{\tilde{p}} := \{ v \in T_{\tilde{p}} \wt{M} \ | \ \| v \| = 1 \}$. 
The function $G \circ \tilde{\gamma} : [0,\infty) \lra \R$ is called the 
{\em radial curvature function} of $(\wt{M}, \tilde{p})$, 
where we denote by $G$ the Gaussian curvature of $\wt{M}$, 
and by $\tilde{\gamma}$ any meridian emanating from 
$\tilde{p} = \tilde{\gamma} (0)$. 
Remark that $f$ satisfies the differential equation 
\[
f''(t) + G (\tilde{\gamma}(t)) f(t) = 0
\]
with initial conditions $f(0) = 0$ and $f'(0) = 1$. For each constant number $\delta >0$, 
a sector $\wt{V} (\delta) \subset \wt{M}$ is defined by   
\[ 
\wt{V} (\delta) : = \left\{ \tilde{x} \in \wt{M} \, | \, 0 < \theta(\tilde{x}) < \delta \right\}_.
\]
Notice that 
the $n$-dimensional model surfaces of revolution are defined similarly, and they are 
completely classified in \cite{KK}.\par 
The total curvature $c (\wt{M})$ of $(\wt{M}, \tilde{p})$ is formally 
defined as the improper integral, i.e., 
\[
c(\wt{M}):=\int _{\wt{M}} G_{+} \circ t\, d\wt{M} +\int_{\wt{M}} G_{-} \circ t \, d\wt{M}
\]
if 
\[
{
\int_{\wt{M}} G_{+}\circ t\, d\wt{M}<\infty, \quad 
{\rm or} \quad \int_{\wt{M}}G_{-}\circ t\,d\wt{M}>-\infty
}_.
\]
Here we set 
\[
G_{+} (t ):=\max\{ G(\tilde{\gamma} (t)), 0 \} 
= \frac{G + |G|}{2}
\]
and
\[
{
G_{-} (t) :=\min\{ G(\tilde{\gamma} (t)), 0 \}
 = \frac{G - |G|}{2}
}_.
\]
Notice that $G = G_{+}\circ t + G_{-}\circ t$. 
If $c(\wt{M})$ exists, $c(\wt{M}) = 2 \pi (1 - \lim_{t \to \infty} f'(t))$ holds, 
since $d\wt{M} = f dt d\theta$ and $f'(0) = 1$. 
By Theorem \ref{CVthm}, 
\[
c(\wt{M}) \le 2 \pi
\]
holds. 
Thus, 
$c (\wt{M}) > - \infty$ means that $\wt{M}$ admits a finite total curvature (if $c(\wt{M})$ exists).\par  
Let $(M,p)$ be a complete non-compact $n$-dimensional Riemannian manifold 
with a base point $p \in M$. We say that $(M, p)$ has 
{\em 
radial curvature at the base point $p$ bounded 
from below by that of 
a non-compact model surface of revolution $(\wt{M}, \tilde{p})$} 
if, along every unit speed minimal geodesic $\gamma: [0,a) \lra M$ 
emanating from $p = \gamma (0)$, 
its sectional curvature $K_M$ satisfies
\[
K_M(\sigma_{t}) \ge G (\tilde{\gamma}(t))
\]
for all $t \in [0, a)$ and all $2$-dimensional linear spaces 
$\sigma_{t}$ spanned by $\gamma'(t)$ 
and a tangent vector to $M$ at $\gamma(t)$. 
Notice that, if the Riemannian metric of $\wt{M}$ is $dt^2 + t^{2}d \theta^2$, or $dt^2 + \sinh^{2} t\,d \theta^2$, 
then $G (\tilde{\gamma}(t)) = 0$, or $G (\tilde{\gamma}(t)) = -1$, respectively.\par 
For this definition, the radial curvature geometry looks artificial, {\bf but this is not the case}, i.e., 
we can construct a model surface of revolution for 
any complete Riemannian manifold with an arbitrary given point 
as a base point (see \cite[Lemma 5.1]{KT2}). 
The existence of a $(\wt{M}, \tilde{p})$ is therefore {\em very natural} on the above definition.

\bigskip

Now, we are in a point where we will state our main theorem\,: 
Let $\cR_{M}$ denote the set of all rays on $M$ and 
$\cR_{p}$ the set of all rays emanating from $p$. 
Moreover, for each $\gamma \in \cR_{M}$, let $\Pi(\gamma)$ denote 
the set of all $\alpha \in \cR_{p}$ which is a limit ray of the sequence of 
minimal geodesic segments joining $p$ to $\gamma  (t_{i})$ for some divergent sequence $\{ t_{i} \}$. 
Hence, $\alpha \in \Pi (\gamma)$ is an asymptotic ray to $\gamma$ emanating from $p$. 
Notice that $\Pi (\gamma) = \{ \gamma \}$, if $\gamma \in \cR_{p}$.\par
We set 
\[
A_{p} := \{\gamma'(0) \in \Sph_{p}^{n - 1} \ | \ \gamma \in \cR_{p}\}, 
\]
where $\Sph^{n - 1}_{p} := \{ v \in T_{p} M \ | \ \| v \| = 1 \}$, and 
denote by $\diam (A_{p})$ the diameter of $A_{p}$. 
A subset $S$ of $A_{p}$ is said to be 
a {\em $\delta$-covering of $A_{p}$}, if 
\[
{
A_{p} \subset \bigcup_{v \in S} \ol{\B_{\delta} (v)}
}_,
\]
where
$
\ol{\B_{\delta} (v)} := 
\left\{ 
w \in \Sph_{p}^{n - 1} \ | \ \angle (v, w) \le \delta 
\right\}
$.

\bigskip

\begin{MT}
Let $(M,p)$ be a complete non-compact connected Riemannian $n$-manifold $M$ 
whose radial curvature at the base point $p$ is bounded from below by
that of a non-compact model surface of revolution $(\wt{M}, \tilde{p})$. 
Assume that 
\begin{enumerate}[{\rm ({MT--}1)}]
\item
$c(\wt{M}) > \pi$, and 
\item
$\wt{M}$ has no pair of cut points in a sector $\wt{V} (\delta_{0})$ for some $\delta_{0} \in (0, \pi]$. 
\end{enumerate}
Then, for any $\gamma_{1}, \gamma_{2}, \ldots, \gamma_{k} \in \cR_{M}$ such that 
$\{\alpha'(0) \in \Sph^{n - 1}_{p} \ | \ \alpha \in \bigcup_{i = 1}^{k} \Pi(\gamma_{i})\}$ is 
a $\delta_{0}$-covering of $A_{p}$, 
\[
\max \{F_{\gamma_{i}} \ | \ i = 1, 2, \ldots, k\}
\]
is an exhaustion. Moreover, 
if 
\[
\diam(A_{p}) \le \delta_{0},
\]
then $F_{\gamma}$ is an exhaustion for all $\gamma \in \cR_{M}$.
\end{MT}

\bigskip\noindent
The property (MT--1) does not always mean that the Gaussian curvature of $\wt{M}$ 
is non-negative everywhere. In fact, the model surface in \cite[Example 1.2]{KT1} 
satisfies both properties (MT--1) and (MT--2), 
but $\lim_{t \to \infty} G\circ \tilde{\gamma}(t) = -\infty$ 
for each meridian $\tilde{\gamma}$.\par  
If a non-compact model surface of revolution $\wt{M}$ admits a finite total curvature, 
then, for each $\ve > 0$, there exists a compact subset $\wt{K}_{\ve}$ of $\wt{M}$ such that 
\[
{
\int_{\wt{M} \setminus \wt{K}_{\ve}} |G|\, d\wt{M}<\ve
}_.
\]
Hence, we might conjecture that the Gaussian curvature of $\wt{M}$ should be almost flat outside 
of a compact subset of $\wt{M}$. The following theorem shows that this conjecture is {\bf false} 
and that the radial curvature function $G(t)$ may change signs wildly.

\begin{theorem}{\rm (\cite{KT4})}\label{thm1.1} 
Let $(\wt{M}, \tilde{p})$ be a non-compact model surface of revolution with its metric (\ref{polar}). 
If $\wt{M}$ admits 
\[
- \infty < c(\wt{M}) < 2 \pi,
\]
then, for any $\ve > 0$, 
there exists a model surface of revolution $(\wh{M}, \wh{p}\,)$ 
with its metric 
\[
\wh{g} 
= 
dt^2 + m(t)^2d \theta^2, \quad 
(t,\theta) \in (0,\infty) \times {\Sph_{\wh{p}}^1}_, 
\]
satisfying the differential equation $m''(t) + \wh{G} (t) m(t) = 0$ 
with initial conditions $m(0) = 0$ and $m'(0) = 1$, 
and admitting a finite total curvature $c(\wh{M})$ such that 
\begin{enumerate}[{\rm (1)}]
\item
$\displaystyle{\left\|\,
G(\tilde{\gamma} (t)) - \wh{G} (t)\,
\right\|_{L_{2}} \le \ve
}$, 
\item
$\displaystyle{
c(\wt{M}) \ge c(\wh{M}) \ge c(\wt{M}) - \ve
}$ 
(respectively 
$\displaystyle{
c(\wt{M}) + \ve \ge c(\wh{M}) \ge c(\wt{M})
}$),
\item
$\displaystyle{
G(\tilde{\gamma} (t)) \ge \wh{G}(t)
}$
(respectively 
$\displaystyle{
\wh{G}(t) \ge G(\tilde{\gamma} (t))
}$) on $[0, \infty)$, and 
\item
$\displaystyle{
\liminf_{t \to \infty} \wh{G} (t) = - \infty
}$
(respectively 
$\displaystyle{
\limsup_{t \to \infty} \wh{G} (t) = \infty}$).
\end{enumerate}
\end{theorem}

\bigskip
The property (MT--2) is satisfied by a von Mangoldt surface of revolution, i.e., 
$\wt{V}(\pi)$ has no pair of cut points. In fact, it was proved in \cite{T} that 
the cut locus of a point on a von Mangoldt surface of revolution 
is empty or a subray of the meridian opposite to the point. 
The assumption (MT--2) is {\bf not strong}. For example, consider a non-compact 
model surface of revolution whose radial curvature function is non-increasing (or non-positive) 
along a {\bf subray} of a meridian. If the surface admits a finite total curvature, 
then the surface admits a sector which has no pair of cut points (see \cite[Sector Theorem]{KT2}). 
We do not know if (MT--2) can be removed from Main Theorem or not.

\bigskip

Since it is clear that $\diam (A_{p}) \le \pi$, as a corollary to Main Theorem, we get  

\begin{theorem}{\rm (\cite[Main Theorem]{KT1})}\label{thm1.4}
Let $(M,p)$ be a complete non-compact Riemannian $n$-manifold $M$ 
whose radial curvature at the base point $p$ is bounded from below by
that of a non-compact von Mangoldt surface of revolution $(M^{*}, p^{*})$.
If $c(M^{*}) > \pi$, then all Busemann functions on $M$ 
are exhaustions.
\end{theorem}

\medskip\noindent 
A related result for Main Theorem is Kasue's \cite[Theorem 4.3]{K}, where he assumed 
that sectional curvature is non-negative, and he controlled diameter of 
each ideal boundary to be less than $\pi / 2$ in his sense. 

\bigskip

In the following sections, 
all geodesics will be normalized, unless otherwise stated.  

\begin{acknowledgement}
The first named author would like to express to Professor S. Ohta his deepest gratitude 
for his helpful comments on the first version of our main theorem 
in the differential topology seminar at Kyoto university, 14th July, 2009. 
\end{acknowledgement}

\section{Mass of Rays on Model Surfaces}\label{sec:mass}
This section is set up as a preliminary to the proof of Main Theorem 
(Theorem \ref{prop2009-05-08}) in the next section. 
Throughout this section, 
let $(\wt{M}, \tilde{p})$ denote a non-compact model surface of revolution which admits 
a total curvature 
$c(\wt{M}) > \pi$.

\medskip

\begin{lemma}\label{lem2009-11-12}
There exists a positive number $r_{1}$ such that 
\[
\int_{V} G \,d\wt{M} > \pi + 2 \Lambda_{0}
\]
holds for all open set $V \subset \wt{M}$ containing $B_{r_{1}} (\tilde{p})$ as a subset. 
Here we set 
\[
\Lambda_{0} := \frac{c(\wt{M}) -\pi}{3}_.
\]
\end{lemma}

\begin{proof}
Since $c(\wt{M})$ is finite, for each positive number $\ve$, 
there exists a positive number $r_{\ve}$ such that 
\[
\int_{\wt{M} \setminus B_{r_{\ve}} (\tilde{p})} |G|\, d\wt{M} < \ve
\]
holds. In particular, for $\ve := \Lambda_{0}$, 
there exists a positive number $r_{1}$ such that 
\begin{equation}\label{lem2009-11-12-1}
\int_{\wt{M} \setminus B_{r_{1}} (\tilde{p})} |G|\, d\wt{M} < \Lambda_{0}
\end{equation}
Let $V \subset \wt{M}$ be an open set containing $B_{r_{1}} (\tilde{p})$ as subset. 
It is clear that 
\begin{align}\label{lem2009-11-12-2}
\int_{V} G\, d\wt{M} 
&\ge 
c(\wt{M}) - \int_{\wt{M} \setminus V} |G|\, d\wt{M}\notag \\[2mm]
&\ge 
c(\wt{M}) -  \int_{\wt{M} \setminus B_{r_{1}} (\tilde{p})} |G|\, d\wt{M}
\end{align}
By (\ref{lem2009-11-12-1}) and (\ref{lem2009-11-12-2}), we get 
\[
{
\int_{V} G \,d\wt{M} >\pi + 2 \Lambda_{0}
}_.
\]
$\qedd$
\end{proof}

Since $c(\wt{M}) > \pi$, it follows from Cohn\,-Vossen's theorem \cite[Satz 5]{CV2} that $\wt{M}$ 
has no straight line. Thus, by \cite[Lemma 6.1.1]{SST}, the next lemma is clear\,:

\begin{lemma}\label{lem2009-08-25-2}
There exists a number 
$r_{2} > r_{1}$ such that no ray emanating from a point in $\wt{M} \setminus B_{r_{2}}(\tilde{p})$
passes through $B_{r_{1}}(\tilde{p})$. 
\end{lemma}

\begin{lemma}\label{lem2009-08-26-1}
For each $\tilde{q} \in \wt{M} \setminus B_{r_{2}}(\tilde{p})$, 
there exists a number $r_{3} > r_{2}$ such that, 
for any $\tilde{x} \in \wt{M} \setminus B_{r_{3}}(\tilde{p})$, 
\[
{
\angle (\tilde{p}\tilde{q}\tilde{x}) \ge \frac{\pi}{2} + \Lambda_{0}
}_.
\]
Here $\angle (\tilde{p}\tilde{q}\tilde{x})$ denotes the angle at the vertex 
$\tilde{q}$ of the geodesic triangle $\triangle(\tilde{p}\tilde{q}\tilde{x})$.
\end{lemma}

\begin{proof}
Take any point $\tilde{q} \in \wt{M} \setminus B_{r_{2}}(\tilde{p})$ and fix it. 
Let $V_{\tilde{q}}$ denote the connected component of 
\[
\wt{M} \ \setminus \bigcup_{\wt{\gamma} \in \cR_{\tilde{q}}} \wt{\gamma} ([0, \infty))
\] 
containing $B_{r_{1}}(\tilde{p})$, where $\cR_{\tilde{q}}$ denotes the set of all rays emanating from 
$\tilde{q}$. 
Notice that the existence of $V_{\tilde{q}}$ is guaranteed by Lemma \ref{lem2009-08-25-2}, 
and that the boundary $\partial V_{\tilde{q}}$ consists of two rays 
$\wt{\alpha}_{+}, \wt{\alpha}_{-} \in \cR_{\tilde{q}}$, which might be the same. 
From Lemma \ref{lem2009-11-12}, 
\[
c(V_{\tilde{q}}) := \int_{V_{\tilde{q}}} G \,d\wt{M} > \pi + 2\Lambda_{0}
\]
holds. On the other hand, since $V_{\tilde{q}}$ does not admit a ray in $\cR_{\tilde{q}}$, 
it follows from \cite[Lemma 6.1.3]{SST} that 
$c(V_{\tilde{q}})$ equals the interior angle at $\tilde{q}$ of $V_{\tilde{q}}$. 
Hence, the interior angle at $\tilde{q}$ of $V_{\tilde{q}}$ is greater than $\pi$. 
Therefore, we get 
\[
\angle(\wt{\alpha}'_{+} (0), \wt{\alpha}'_{-}(0))  = 2 \pi - c(V_{\tilde{q}}) < \pi - 2 \Lambda_{0}.
\]
Since $V_{\tilde{q}}$ does not admit a ray in $\cR_{\tilde{q}}$ 
and $\wt{\alpha}_{+}$, $\wt{\alpha}_{-}$ are symmetric under the reflection with respect to 
the meridian $\mu_{\tilde{q}}$ passing through $\tilde{q}$, 
\begin{align}\label{lem2009-08-26-1-1}
\max \{ 
\angle(\wt{\gamma}'(0),  \mu'_{\tilde{q}} (d(\tilde{p}, \tilde{q})) ) \ | \ 
\wt{\gamma} \in \cR_{\tilde{q}}
\} 
&=
\angle(\wt{\alpha}'_{+} (0), \mu'_{\tilde{q}} (d(\tilde{p}, \tilde{q})) )\notag \\[2mm] 
&= 
\angle(\wt{\alpha}'_{-}(0),  \mu'_{\tilde{q}} (d(\tilde{p}, \tilde{q})) )\notag \\[2mm] 
&< 
{\frac{\pi}{2} - \Lambda_{0}}_.
\end{align}
In particular, by (\ref{lem2009-08-26-1-1}), 
\[
\angle(\wt{\gamma}'(0),  \mu'_{\tilde{q}} (d(\tilde{p}, \tilde{q})) ) < \frac{\pi}{2} - \Lambda_{0}
\]
holds for all $\wt{\gamma} \in \cR_{\tilde{q}}$.\par
Let $\wt{\alpha} : [0, d(\tilde{q}, \tilde{x})] \lra \wt{M}$ denote a minimal geodesic segment joining 
$\tilde{q}$ to a point $\tilde{x} \in \wt{M}$. 
If $d(\tilde{q}, \tilde{x})$ is sufficient large, 
then $\wt{\alpha}'(0)$ is close to some $\wt{\gamma}'(0)$, $\wt{\gamma} \in \cR_{\tilde{q}}$. 
Therefore, there exists a number $r_{3} > r_{2}$ such that, for any minimal geodesic segment 
$\wt{\alpha} : [0, d(\tilde{q}, \tilde{x})] \lra \wt{M}$ joining $\tilde{q}$ to $\tilde{x}$ with 
$d(\tilde{q}, \tilde{x}) > r_{3}$, 
\begin{equation}\label{lem2009-08-26-1-2}
{
\angle(\wt{\alpha}'(0),  \mu'_{\tilde{q}} (d(\tilde{p}, \tilde{q})) ) 
< \frac{\pi}{2} - \Lambda_{0}
}_.
\end{equation}
The equation (\ref{lem2009-08-26-1-2}) implies that 
\[
\angle (\tilde{p}\tilde{q}\tilde{x}) \ge \frac{\pi}{2} + \Lambda_{0}
\]
for all $\tilde{x} \in \wt{M} \setminus B_{r_{3}}(\tilde{p})$. 
$\qedd$
\end{proof}

\section{Proof of Main Theorem}\label{sec:ET}
Our purpose of this section is to prove Main Theorem (Theorem \ref{prop2009-05-08}). 
In the proof of the theorem, we will apply a new type of the Toponogov comparison theorem. 
The comparison theorem was established by the present authors 
as generalization of the comparison theorem in conventional comparison geometry, 
which is stated as follows\,:

\medskip

\begin{TCT}\label{TCT}\hspace{-1.5mm}{\rm (\cite[Theorem 4.12]{KT2})}\ \par
Let $(M,p)$ be a complete non-compact Riemannian manifold $M$ 
whose radial curvature at the base point $p$ is bounded from below by
that of a non-compact model surface of revolution $(\wt{M}, \tilde{p})$. 
If $(\wt{M}, \tilde{p})$ admits a sector $\wt{V}(\delta_{0})$, 
$\delta_{0} \in (0, \pi]$, having no pair of cut points, 
then, for every geodesic triangle $\triangle(pxy)$ in $(M,p)$ 
with $\angle (xpy) < \delta_{0}$, 
there exists a geodesic triangle 
$\wt{\triangle} (pxy) :=\triangle(\tilde{p}\tilde{x}\tilde{y})$ 
in $\wt{V}(\delta_{0})$ such that
\begin{equation}\label{TCT-length}
d(\tilde{p},\tilde{x})=d(p,x), \quad d(\tilde{p},\tilde{y})=d(p,y), \quad d(\tilde{x},\tilde{y})=d(x,y) 
\end{equation}
and that
\begin{equation}\label{TCT-angle}
\angle (xpy) \ge \angle (\tilde{x}\tilde{p}\tilde{y}), \quad  
\angle (pxy) \ge \angle (\tilde{p}\tilde{x}\tilde{y}), \quad
\angle (pyx) \ge \angle (\tilde{p}\tilde{y}\tilde{x}). 
\end{equation}
Here $\angle(pxy)$ denotes the angle between the minimal geodesic segments 
from $x$ to $p$ and $y$ forming the triangle $\triangle(pxy)$.
\end{TCT}

\begin{remark}
In \cite{KT3}, the present authors very recently generalized, from the radial curvature 
geometry's standpoint, the Toponogov comparison theorem to a complete Riemannian manifold 
with smooth convex boundary.
\end{remark}

\bigskip

Hereafter, let $(M,p)$ denote 
a complete non-compact Riemannian $n$-manifold $M$ whose radial curvature 
at the base point $p$ is bounded from below by
that of a non-compact model surface of revolution $(\wt{M}, \tilde{p})$ with its metric (\ref{polar}), 
$\cR_{M}$ the set of all rays on $M$, and $\cR_{p}$ the set of all rays emanating from $p$. 
Moreover, for each $\gamma \in \cR_{M}$, let $\Pi(\gamma)$ denote the set of all 
$\alpha \in \cR_{p}$ which is a limit ray of the sequence of minimal geodesic 
segments joining $p$ to $\gamma  (t_{i})$ for some divergent sequence $\{ t_{i} \}$. 
Furthermore, we assume that 
\begin{enumerate}[{\rm ({MTI--}1)}]
\item
$c(\wt{M}) > \pi$, and 
\item
$\wt{M}$ has no pair of cut points in a sector $\wt{V} (\delta_{0})$ for some $\delta_{0} \in (0, \pi]$. 
\end{enumerate}

\begin{lemma}\label{lem2009-04-26-1}
Let $\gamma \in \cR_{M}$ and $\alpha : [0, d(p, q)] \lra M$ a minimal geodesic segment 
joining $p$ to a point $q \in M \setminus B_{r_{2}} (p)$ such that 
\[
\angle (\alpha'(0), \beta_{\gamma}'(0)) < \delta_{0}
\]
for some $\beta_{\gamma} \in \Pi (\gamma)$. Then, 
\[
\angle(\sigma'(0), \alpha'(d(p, q))) \le \frac{\pi}{2} - \Lambda_{0}
\]
holds for a ray $\sigma$ emanating from $q$ asymptotic to $\gamma$. Here 
$\Lambda_{0}$ and $r_{2}$ denote the positive numbers guaranteed in Lemmas 
\ref{lem2009-11-12} and \ref{lem2009-08-25-2}, respectively.
\end{lemma}

\begin{proof}
Since $\beta_{\gamma} \in \Pi (\gamma)$, there exists a divergent sequence $\{t_{i}\}$ 
such that the sequence of minimal geodesic segments 
$\beta_{i} : [0, d(p, \gamma (t_{i}))] \lra M$ joining $p$ to $\gamma (t_{i})$ convergent to 
$\beta_{\gamma}$. 
Since 
$\lim_{t \to 0}\angle (\beta'_{i}(0),  \beta_{\gamma}'(0)) = 0$,
there is a number $i_{0} \in \N$ such that 
\[
\angle (\beta'_{i}(0),  \alpha'(0)) < \delta_{0}
\]
for all $i \ge i_{0}$. Thus, by the new type of the Toponogov comparison theorem, 
there exists a geodesic triangle 
$\wt{\triangle} (p \gamma(t_{i}) q) \subset \wt{V}(\delta_{0})$ 
corresponding to the triangle $\triangle (p \gamma(t_{i}) q)$, $i \ge i_{0}$, 
such that (\ref{TCT-length}) holds for $x = \gamma(t_{i})$ and $y = q$, and that 
\[
\angle (- \alpha' (d(p,q)), \sigma'_{i} (0)) \ge 
\angle (\tilde{p}\tilde{q}\tilde{\gamma}(t_{i})).
\]
Here $\sigma_{i} : [0, d(q, \gamma (t_{i}))] \lra M$ denotes a minimal geodesic segment 
joining $q$ to $ \gamma(t_{i})$. 
By Lemma \ref{lem2009-08-26-1}, we get 
\[
\angle (- \alpha' (d(p,q)), \sigma'_{i} (0)) \ge \frac{\pi}{2} + \Lambda_{0}
\]
for sufficiently large $i$. Hence, 
\[
\angle (- \alpha' (d(p,q)), \sigma'(0)) \ge \frac{\pi}{2} + \Lambda_{0}
\]
where $\sigma$ denotes a limit ray of the sequence $\{\sigma_{i}\}$, 
which is asymptotic to $\gamma$.
$\qedd$
\end{proof}

Hereafter, let $F_{\gamma}$ denote a Busemann function of a $\gamma \in \cR_{M}$. 
Notice that, by the definition of $F_{\gamma}$, 
$|F_{\gamma}(x) - F_{\gamma}(y)| \le d(x, y)$ 
holds for all $x, y \in M$, i.e., $F_{\gamma}$ is Lipschitz continuous with Lipschitz constant $1$. 
Hence, $F_{\gamma}$ is differentiable except for a measure zero set. 
Moreover, we have

\begin{proposition}{\rm (\cite[Theorem 3.1]{KT1})}\label{prop2009-05-07}
Let $\gamma$ be a ray on a complete non-compact Riemannian manifold $M$. 
Then, $F_{\gamma}$ is differentiable at a point $q \in M$ 
if and only if there exists a unique ray emanating from $q$ asymptotic to $\gamma$. 
Moreover, the gradient vector of
$F_{\gamma}$ at a differentiable point $q$ equals the velocity vector of
the unique ray asymptotic to $\gamma$. 
\end{proposition}

\begin{lemma}\label{lem2009-05-07-1}
Let $\gamma \in \cR_{M}$ and $\alpha : [0, d(p, q)] \lra M$ a minimal geodesic segment 
joining $p$ to a point $q \in M \setminus B_{r_{2}} (p)$ such that 
\[
\angle (\alpha'(0), \beta_{\gamma}'(0)) < \delta_{0}
\]
for some $\beta_{\gamma} \in \Pi (\gamma)$. 
If $F_\gamma$ is differentiable at $\alpha(t)$ for almost all $t \in (a, b) \subset (r_{2}, d(p, q)]$, 
then
\[
F_\gamma(\alpha(b))-F_\gamma(\alpha(a)) \ge (b - a) \sin \Lambda_{0}.
\]
\end{lemma}

\begin{proof}
Assume that $F_\gamma$ is differentiable at $\alpha(t_{0})$, $t_{0} \in (a, b)$. 
By Lemma \ref{lem2009-04-26-1} and Proposition \ref{prop2009-05-07}, 
we get 
\[
\angle ((\nabla F_\gamma)_{\alpha(t_{0})}, \alpha'(t_{0})) 
\le \frac{\pi}{2} - \Lambda_{0}
\]
Hence, for almost all $t\in(a, b)$, 
\[
\frac{d}{dt}F_{\gamma}(\alpha(t)) 
= \langle (\nabla F_\gamma)_{\alpha(t)}, \alpha'(t) \rangle 
= \cos \left( \angle ((\nabla F_\gamma)_{\alpha(t)}, \alpha'(t)) \right) 
\ge \sin \Lambda_{0}
\]
It follows from Dini's theorem \cite{D} 
(cf.\,\cite[Section 2.3]{H}, \cite[Theorem 7.29]{WZ}) that  
\[
F_\gamma(\alpha(b))-F_\gamma(\alpha(a)) 
= \int^{b}_{a} \frac{d}{dt}F_{\gamma}(\alpha(t)) \, dt 
\ge (b -a) \sin \Lambda_{0}. 
\]
$\qedd$
\end{proof}

\begin{lemma}\label{lem2009-05-07-2}
Let $\gamma \in \cR_{M}$ and $\alpha : [0, d(p, q)] \lra M$ a minimal geodesic segment 
joining $p$ to a point $q \in M \setminus B_{r_{2}} (p)$ such that 
\[
\angle (\alpha'(0), \beta_{\gamma}'(0)) \le \delta_{0}
\]
for some $\beta_{\gamma} \in \Pi (\gamma)$. 
Then, 
\begin{equation}\label{lem2009-05-07-2-A}
F_\gamma(q)-F_\gamma(\alpha(r_{2})) \ge \left( d(p,q) - r_{2} \right) \sin \Lambda_{0}
\end{equation}
holds. 
\end{lemma}

\begin{proof} 
First, we will prove (\ref{lem2009-05-07-2-A}) under the assumption that 
\[
\angle (\alpha'(0), \beta_{\gamma}'(0)) < \delta_{0}.
\]
The general case will be completed by the limit argument. 
If we prove that, for each $t_0 \in(r_{2}, d(p, q))$, 
there exists a number $\ve_0 > 0$ such that
\begin{equation}\label{lem2009-05-07-2-1}
F_{\gamma}(\alpha(t))-F_{\gamma}(\alpha(s)) \geq (t - s ) \sin \Lambda_{0}
\end{equation}
holds for all $s, t \in (t_{0} - \ve_{0}, t_{0} + \ve_{0})$ with $s < t$, 
then the equation (\ref{lem2009-05-07-2-A}) is clear.\par
Take any $t_{0} \in (r_{1}, d(p, q))$, and fix it. 
Since $\alpha$ is minimal on $[0, d(p, q)]$, 
$\alpha(t_0)$ is not a cut point of $p = \alpha(0)$. 
Hence, there exist an open neighborhood $\cU \subset \Sph^{n - 1}_{p}$ around $\alpha' (0)$, 
an open neighborhood $U$ around $\alpha(t_0)$, and 
an open interval $(t_0 - \ve_0, t_0 + \ve_0)$ such that
$\cU \times (t_0 - \ve_0, t_0 + \ve_0)$ is diffeomorphic to $U$ 
by a map $\varphi$, where $\varphi^{-1}(v, t) := \exp_{p} (t v)$.  
Since $F_\gamma \circ \varphi^{-1}$
is Lipschitz, it follows from  Rademacher's theorem (cf.\,\cite{Mo}) 
that there exists a set $\cE \subset T_{p}M$ of Lebesgue measure zero such that 
$F_\gamma \circ \varphi^{-1}$ is differentiable on 
$\left( \cU \times (t_0 - \ve_0, t_0 + \ve_0) \right) \setminus \cE$. 
Moreover,
for each $v \in \cU$, we set 
\[
\cE_{v} := \{ t \in (t_0 - \ve_0, t_0 + \ve_0) \, | \, (v, t) \in \cE \}_.
\]
Remark that the set $\cE_{v} $ has also 
Lebesgue measure zero for all most all $v \in \cU$ (cf.\,\cite[Lemma 6.5]{WZ}). 
Thus, we may find a sequence $\{ \alpha_{j} \}$ of minimal geodesic segments emanating
from $p$ converging to $\alpha$ such that 
each $F_{\gamma}$ is differentiable at $\alpha_{j} (t)$ for almost all  
$t \in (t_{0} - \ve_0, t_{0} + \ve_0)$. 
By Lemmas \ref{lem2009-04-26-1} and \ref{lem2009-05-07-1}, for each $j \in\N$, 
\[
F_{\gamma}(\alpha_{j} (t) ) - F_{\gamma}(\alpha_{j} (s) ) \ge ( t - s ) \sin \Lambda_{0}
\]
holds for all $s, t \in (t_{0} - \ve_0, t_{0} + \ve_0)$ with $s < t$.
Then, by taking the limit, we get (\ref{lem2009-05-07-2-1}).\par

Assume that 
\[
\angle (\alpha'(0), \beta_{\gamma}'(0)) = \delta_{0}.
\]
It is clear that there exists a sequence $\{\alpha_{i} : [0, \ell_{i}] \lra M \}$ of 
minimal geodesic segments $\alpha_{i}$ emanating from $p = \alpha_{i} (0)$ convergent to $\alpha$ 
such that 
$\angle (\alpha_{i}'(0), \beta_{\gamma}'(0)) < \delta_{0}$ for each $i \in \N$. 
From the argument above, 
\[
F_{\gamma}(\alpha_{i} (\ell_{i}) ) - F_{\gamma}(\alpha_{i} (r_{2}) ) \ge ( \ell_{i} - r_{2} ) \sin \Lambda_{0}
\]
By taking the limit, we get (\ref{lem2009-05-07-2-A}).
$\qedd$
\end{proof}

Set 
\[
A_{p} := \{\gamma'(0) \in \Sph_{p}^{n - 1} \ | \ \gamma \in \cR_{p}\},
\]
and denote by $\diam (A_{p})$ the diameter of $A_{p}$. 
Then, we have our main theorem in this article\,:

\begin{theorem}\label{prop2009-05-08}
For any $\gamma_{1}, \gamma_{2}, \ldots, \gamma_{k} \in \cR_{M}$ such that 
$\{\alpha'(0) \in \Sph^{n - 1}_{p} \ | \ \alpha \in \bigcup_{i = 1}^{k} \Pi(\gamma_{i})\}$ is 
a $\delta_{0}$-covering of $A_{p}$,
\[
\max \{F_{\gamma_{i}} \ | \ i = 1, 2, \ldots, k\}
\]
is an exhaustion. Moreover, 
if $\diam(A_{p}) \le \delta_{0}$, or $\delta_{0} = \pi$, 
then $F_{\gamma}$ is an exhaustion for all $\gamma \in \cR_{M}$.
\end{theorem}

\begin{proof}
Suppose that $\max \{F_{\gamma_{i}} \ | \ i = 1, 2, \ldots, k\}$ is not an exhaustion, i.e., 
for some $a \in \R$, 
\[
X := \bigcap_{i = 1}^{k} F_{\gamma_{i}}^{-1} (-\infty, a]
\]
is non-compact. 
Hence, there exists a sequence $\{q_{j}\}$ of points $q_{j} \in X$ such that 
\[
{
\lim_{j \to \infty} d(p, q_{j}) = \infty
}_.
\]
Let $\alpha_{j} : [0, d(p, q_{j})] \lra M$ denote a minimal geodesic segment joining 
$p$ to $q_{j}$. 
Since $\lim_{j \to \infty} d(p, q_{j}) = \infty$, there exists a number $j_{0} \in \N$ such that 
\[
r_{2} < d(p, q_{j})
\]
for all $j \ge j_{0}$. 
Furthermore, by choosing an infinite subsequence of $\{\alpha_{j}\}$, 
we may assume that there exist $i_{0} \in \{1, 2, \ldots, k\}$ such that, 
for each $j \ge j_{0}$, 
\[
\angle (\alpha'_{j} (0), \beta_{\gamma_{j}}'(0) ) \le \delta_{0}
\]
holds for some $\beta_{\gamma_{j}} \in \Pi (\gamma_{i_{0}})$. 
It follows from Lemma \ref{lem2009-05-07-2} that 
\[
F_{\gamma_{i_{0}}} (q_{j}) - F_{\gamma_{i_{0}}} (\alpha_{j}(r_{2})) 
\ge \left( d(p,q_{j}) - r_{2} \right) \sin \Lambda_{0}
\]
for all $j \ge j_{0}$. 
Since $q_{j} \in F_{\gamma_{i_{0}}}^{-1} (-\infty, a]$ for all $j \ge j_{0}$, 
\[
a - F_{\gamma_{i_{0}}} (\alpha_{j}(r_{2})) 
\ge \left( d(p,q_{j}) - r_{2} \right) \sin \Lambda_{0}.
\] 
Since $\lim_{j \to \infty} d(p, q_{j}) = \infty$, we have 
$\lim_{j \to \infty} F_{\gamma_{i_{0}}} (\alpha_{j}(r_{2})) =  - \infty$. 
This is impossible, since 
$|F_{\gamma_{i_{0}}} (p) - F_{\gamma_{i_{0}}} (\alpha_{j}(r_{2}))| \le d(p, \alpha_{j}(r_{2})) = r_{2}$ 
for all $j \ge j_{0}$. 
Therefore, 
$\max \{F_{\gamma_{i}} \ | \ i = 1, 2, \ldots, k\}$ is an exhaustion.\par
Next, we will prove the second claim. Assume that $\diam(A_{p}) \le \delta_{0}$. 
Since $\angle (v, w) \le \delta_{0}$ for all $v, w \in A_{p}$, it is clear that $\{v\}$ is 
a $\delta_{0}$-covering of $A_{p}$ for each $v \in A_{p}$. 
Hence, for each $\gamma \in \cR_{M}$, 
$\{\alpha'(0) \in \Sph^{n - 1}_{p} \ | \ \alpha \in \Pi(\gamma)\}$ is 
a $\delta_{0}$-covering of $A_{p}$. 
From the argument above, 
this implies that $F_{\gamma}$ 
is an exhaustion for all $\gamma \in \cR_{M}$. If $\delta_{0} = \pi$, then the claim is clear, 
since $\diam(A_{p}) \le \pi$.
$\qedd$
\end{proof}

From the same argument in \cite[Section 4.2]{KT1}, we get

\begin{corollary}\label{cor2009-05-08-2}
The isometry group $I(M)$ of $M$ is compact, 
if $\diam (A_{p}) \le \delta_{0}$, or $\delta_{0} = \pi$. 
\end{corollary}

\bigskip

\medskip

\begin{center}
Kei KONDO $\cdot$ Minoru TANAKA 

\bigskip
Department of Mathematics\\
Tokai University\\
Hiratsuka City, Kanagawa Pref.\\ 
259\,--\,1292 Japan

\bigskip

{\small
$\bullet$\,our e-mail addresses\,$\bullet$

\bigskip 
\textit{e-mail of Kondo} \,:

\medskip
{\tt keikondo@keyaki.cc.u-tokai.ac.jp}

\medskip
\textit{e-mail of Tanaka}\,:

\medskip
{\tt m-tanaka@sm.u-tokai.ac.jp}
}
\end{center}


\begin{thebibliography}{GMST}

\bibitem[CG]{CG}
J.~Cheeger and D.~Gromoll, 
{\it On the structure of complete manifolds of nonnegative curvature}, 
Ann.\ of \ Math. 
\textbf{96} (1972), 415--443.

\bibitem[CV1]{CV1}
S.~Cohn\,-Vossen, 
{\it K\"urzeste Wege und Totalkr\"ummung auf Fl\"achen}, 
Compositio Math. \textbf{2} (1935), 63--113.

\bibitem[CV2]{CV2}
S.~Cohn\,-Vossen, 
{\it Totalkr\"ummung und geod\"atische Linien auf einfach zusammenh\"angenden 
offenen volst\"andigen Fl\"achenst\"ucken}, 
Recueil Math. Moscow \textbf{43} (1936), 139--163.

\bibitem[D]{D}
U.~Dini, {\it Fondamenti per la teorica delle funzioni di variabili reali},
Pisa, (1878).

\bibitem[GM]{GM}
D.~Gromoll and W.~Meyer, {\it On complete manifolds of positive curvature},
Ann.\ of Math.\ $(2)$ \textbf{75} (1969), 75--90.

\bibitem[H]{H}
P.~Hartman, {\it Geodesic parallel coordinates in the large},
Amer. J. Math. \textbf{86} (1964), 705--727.

\bibitem[Ha]{Ha}
T.~Hawkins, Lebesgue's Theory of Integration : Its origins and development, 
University of Wisconsin Press, Madison 1970.

\bibitem[K]{K}
A.~Kasue, {\it A compactification of a manifold with asymptotically
nonnegative curvature}, Ann. Sci. Ecole Norm, Sup. (4) \textbf{21} (1988), no. 4, 593--622

\bibitem[KK]{KK}
N.\,N.~Katz and K.~Kondo, 
{\it Generalized space forms}, 
Trans. Amer.\ Math.\ Soc. \textbf{354} (2002), 2279--2284. 

\bibitem[KT1]{KT1}
K.~Kondo and M.~Tanaka, 
{\it Total curvatures of model surfaces control 
topology of complete open manifolds with radial curvature bounded below.\,I}, 
to appear in Mathematische Annalen. 
The online first version was published (13 November 2010). 
DOI 10.1007/s00208-010-0593-4 


\bibitem[KT2]{KT2}
K.~Kondo and M.~Tanaka, 
{\it Total curvatures of model surfaces control 
topology of complete open manifolds with radial curvature bounded below.\,II}, 
Trans. Amer.\ Math.\ Soc. \textbf{362} (2010), 6293--6324. 

\bibitem[KT3]{KT3}
K.~Kondo and M.~Tanaka, 
{\it Toponogov comparison theorem for open triangles}, 
Preprint 2009, {\tt http://arxiv.org/abs/0905.3236}

\bibitem[Mo]{Mo}
F.~Morgan, Geometric Measure Theory, A Beginer's Guide, 
Academic Press, 1988.

\bibitem[S]{Sh1}
K.~Shiohama, {\it The role of total curvature on complete noncompact Riemannian $2$-manifolds},
Illinois J.\ Math. \textbf{28} (1984), 597--620.

\bibitem[SST]{SST}
K.~Shiohama, T.~Shioya, and M.~Tanaka, 
The Geometry of Total Curvature on Complete Open Surfaces, 
Cambridge tracts in mathematics \textbf{159}, 
Cambridge University Press, Cambridge, 2003.

\bibitem[T]{T}
M.~Tanaka, {\it On the cut loci of a von Mangoldt's surface of revolution},
J.\ Math.\ Soc.\ Japan \textbf{44} (1992), 631--641.

\bibitem[TK]{KT4}
M.~Tanaka and K.~Kondo, 
{\it The Gauss curvature of a model surface with 
finite total curvature is not always bounded.}, {\tt http://arxiv.org/abs/1102.0852}

\bibitem[To]{To}
V.\,A.~Toponogov, {\it Riemannian spaces containing straight lines} (in Russian), 
Dokl.\ Akad.\ Nauk \ SSSR \textbf{127} (1959), 977--979.

\bibitem[WZ]{WZ}
R.\,L.~Wheeden and A.~Zygmund, Measure and Integral, Marcel Decker, New York, 1977.

\end{thebibliography}
\end{document}